\documentclass[12pt]{article}
\usepackage{graphicx}
\usepackage{amsmath,amsthm,amssymb,enumerate}
\usepackage{euscript,mathrsfs}
\usepackage{color}
\usepackage{dsfont}
\usepackage[left=2cm,right=2cm,top=4cm,bottom=4cm]{geometry}
\usepackage{color}
\usepackage[framemethod=tikz]{mdframed}
\usepackage{bm}
\allowdisplaybreaks

\usepackage{esint}
\usepackage{soul}

\catcode`\@=11 \@addtoreset{equation}{section}

\catcode`\@=12

\allowdisplaybreaks

\newtheorem{Theorem}{Theorem}[section]
\newtheorem{Proposition}[Theorem]{Proposition}
\newtheorem{Lemma}[Theorem]{Lemma}
\newtheorem{Corollary}[Theorem]{Corollary}

\theoremstyle{definition}
\newtheorem{Definition}[Theorem]{Definition}

\newtheorem{Remark}[Theorem]{Remark}

\newcommand{\bTheorem}[1]{
\begin{Theorem} \label{T#1} }
\newcommand{\eT}{\end{Theorem}}

\newcommand{\bProposition}[1]{
\begin{Proposition} \label{P#1}}
\newcommand{\eP}{\end{Proposition}}

\newcommand{\bLemma}[1]{
\begin{Lemma} \label{L#1} }
\newcommand{\eL}{\end{Lemma}}

\newcommand{\bCorollary}[1]{
\begin{Corollary} \label{C#1} }
\newcommand{\eC}{\end{Corollary}}

\newcommand{\bRemark}[1]{
\begin{Remark} \label{R#1} }
\newcommand{\eR}{\end{Remark}}

\newcommand{\bDefinition}[1]{
\begin{Definition} \label{D#1} }
\newcommand{\eD}{\end{Definition}}

\newcommand{\Ds}{\mathbb{D}_x}

\newcommand{\vre}{\rho^{\ep}}
\newcommand{\vue}{\bu^{\ep}}

\newcommand{\bfphi}{\boldsymbol{\varphi}}

\newcommand{\bFormula}[1]{
\begin{equation} \label{#1}}
\newcommand{\eF}{\end{equation}}

\newcommand{\Ov}[1]{\overline{#1}}

\newcommand{\DC}{C^\infty_c}
\newcommand{\aleq}{\stackrel{<}{\sim}}

\newcommand{\vr}{\rho}

\newcommand{\Div}{{\rm div}_x}
\newcommand{\Grad}{\nabla_x}

\newcommand{\dx}{\,{\rm d} {x}}

\newcommand{\dt}{\,{\rm d} t }

\newcommand{\intO}[1]{\int_{\Omega} #1 \ \dx}

\newcommand{\intRd}[1]{\int_{\R^d} #1 \ \dx}

\newcommand{\D}{{\rm d}}

\newcommand{\R}{\mathbb{R}}

\newcommand{\br}{ \nonumber \\ }

\newcommand{\mc}{\mathcal}
\newcommand{\vc}{\mathbf}
\newcommand{\ep}{\varepsilon}
\newcommand{\bu}{\mathbf{u}}
\newcommand{\bphi}{\bm{\varphi}}
\renewcommand{\leq}{\leqslant}
\renewcommand{\geq}{\geqslant}

\def\softd{{\leavevmode\setbox1=\hbox{d}%
          \hbox to 1.05\wd1{d\kern-0.4ex{\char039}\hss}}}
\definecolor{Cgrey}{rgb}{0.85,0.85,0.85}
\definecolor{Cblue}{rgb}{0.50,0.85,0.85}
\definecolor{Cred}{rgb}{1,0,0}
\definecolor{fancy}{rgb}{0.10,0.85,0.10}

\newcommand\Cbox[2]{%
    \newbox\contentbox%
    \newbox\bkgdbox%
    \setbox\contentbox\hbox to \hsize{%
        \vtop{
            \kern\columnsep
            \hbox to \hsize{%
                \kern\columnsep%
                \advance\hsize by -2\columnsep%
                \setlength{\textwidth}{\hsize}%
                \vbox{
                    \parskip=\baselineskip
                    \parindent=0bp
                    #2
                }%
                \kern\columnsep%
            }%
            \kern\columnsep%
        }%
    }%
    \setbox\bkgdbox\vbox{
        \color{#1}
        \hrule width  \wd\contentbox %
               height \ht\contentbox %
               depth  \dp\contentbox
        \color{black}
    }%
    \wd\bkgdbox=0bp%
    \vbox{\hbox to \hsize{\box\bkgdbox\box\contentbox}}%
    \vskip\baselineskip%
}

\mdfdefinestyle{MyFrame}{%
	linecolor=black,
	outerlinewidth=1pt,
	roundcorner=5pt,
	innertopmargin=\baselineskip,
	innerbottommargin=\baselineskip,
	innerrightmargin=10pt,
	innerleftmargin=10pt,
	backgroundcolor=white!20!white}

\date{}

\makeindex
\begin{document}


\title{On the motion of a large number of small rigid bodies in a viscous incompressible fluid}

\author{Eduard Feireisl
	\thanks{The work of E.F. was partially supported by the
		Czech Sciences Foundation (GA\v CR), Grant Agreement
		21--02411S. The Institute of Mathematics of the Academy of Sciences of
		the Czech Republic is supported by RVO:67985840. A.R and A.Z have been partially supported by the Basque Government through the BERC 2022-2025 program and by the Spanish State Research Agency through BCAM Severo Ochoa excellence accreditation SEV-2017-0718 and through project PID2020-114189RB-I00 funded by Agencia Estatal de Investigación (PID2020-114189RB-I00 / AEI / 10.13039/501100011033). A.Z. was also partially supported  by a grant of the Ministry of Research, Innovation and Digitization, CNCS - UEFISCDI, project number PN-III-P4-PCE-2021-0921, within PNCDI III. The research of A.R. has been supported by the Alexander von Humboldt-Stiftung / Foundation.} 
	\and Arnab Roy$^1$ \and Arghir Zarnescu$^{2,3,4}$
}

\date{\today}

\maketitle

\bigskip

\centerline{$^*$  Institute of Mathematics of the Academy of Sciences of the Czech Republic,}

\centerline{\v Zitn\' a 25, CZ-115 67 Praha 1, Czech Republic}

\centerline{$^1$ Technische Universit\"{a}t Darmstadt,}

\centerline{Schlo\ss{}gartenstra{\ss}e 7, 64289 Darmstadt, Germany.}

\centerline{$^2$ BCAM, Basque Center for Applied Mathematics,}

\centerline{Mazarredo 14, E48009 Bilbao, Bizkaia, Spain}

\centerline{$^3$IKERBASQUE, Basque Foundation for Science, }

\centerline{Plaza Euskadi 5, 48009 Bilbao, Bizkaia, Spain}

\centerline{$^4$`Simion Stoilow" Institute of the Romanian Academy,}

\centerline{21 Calea Grivi\c{t}ei, 010702 Bucharest, Romania }

\maketitle

\begin{abstract}
	
	We consider the motion of $N$ rigid bodies -- compact sets $(\mathcal{S}^1_\varepsilon, \cdots, \mathcal{S}^N_\varepsilon )_{\varepsilon > 0}$ --  immersed in a viscous incompressible fluid 
	contained in a {domain in} the Euclidean space $\mathbb{R}^d$, $d=2,3$. 
	We show the fluid flow is not influenced by the presence of the infinitely many bodies in the asymptotic limit $\varepsilon \to 0$ and $N=N(\varepsilon)\rightarrow\infty$ 
	as soon as 
	\[ 
	{\rm diam}[\mathcal{S}^i_\varepsilon ] \to 0 \ \mbox{as}\ \varepsilon \to 0 ,\ i=1,\cdots, N(\varepsilon).
	\]
	The result depends solely on the geometry of the bodies and is independent of their mass densities. Collisions are allowed and the initial data are arbitrary with finite energy.

\end{abstract}

\bigskip


\section{Introduction}
\label{i}

There is a number of studies concerning the impact of a small rigid body immersed in a viscous fluid on the fluid motion. A general approach used so far is based on the idea that if the body is small but ``heavy'', meaning its mass density $\rho^{\ep}_{\mc{S}}$ is large, its velocity can be controlled and the resulting situation is therefore close to the rigid obstacle problem. He and Iftimie \cite{HeIft1}, \cite{HeIft2} exploited this idea to handle the case when the body mass density satisfies 
$\rho^{\ep}_{\mc{S}} \to \infty$, while the rigid body diameter is proportional to a small number $\ep$. More recently, 
Bravin and Ne\v casov\' a \cite{BraNec2} showed that if the density is ``very large'', the rigid object keeps moving with its initial velocity not being influenced by the fluid. Note that these results are slightly at odds with 
a physically relevant hypothesis that the body density should be at least bounded and also with a commonly accepted 
scenario that a \emph{light} particle should not have any major impact on the fluid motion. 

To the best of our knowledge, the only available result concerning a body with a constant density 
was obtained by Lacave and Takahashi \cite{LacTak} in the case of the planar motion. Their technique, similarly to a
more recent paper by Tucsnak et al. \cite{SDM2022}, is based on the $L^p-L^q$ theory for the associated solution 
semigroup and requires smallness of the initial fluid velocity. The authors also note that the result can be extended to the case of several ``massive'' bodies up to the first contact. 

Our goal is to extend the main result of 
\cite{LacTak} to the case of several bodies 
$(\mathcal{S}^1_\ep, \cdots , \mathcal{S}^N_\ep)_{\ep > 0}$
that may collide in the evolution process. More specifically, our result holds under the following hypotheses:

\begin{itemize}
	
	\item The fluid is confined to an arbitrary domain $\Omega \subset \mathbb{R}^d$, d=2,3. 
	
	\item The mass density of the bodies is irrelevant. The only restriction concerns their shape. 
	Specifically, we suppose  
	\begin{align} 
		D_\ep &\equiv \max_{i=1, \cdots, N} \{ {\rm diam}[ \mathcal{S}^i_\ep ] \} \to 0 \ \mbox{as}\ \ep \to 0, \br
		0 < \lambda D_\ep^\beta &\leq |\mathcal{S}_\ep |  \ \mbox{as}\ \ep \to 0,\ 
		d {\leq } \beta < \left\{ \begin{array}{l} 15 \ \mbox{if}\  d = 3, \\ \\ 
	\mbox{arbitary finite if}\ d = 2, \end{array} \right. \label{cond}
		\end{align}
	{for some $\lambda > 0$ independent of $\ep$}.
	
	\item The result is global in time and holds in the class of weak solutions and for any finite--energy initial data.
	
	\end{itemize}

Very roughly indeed, we may conclude that the effect of a finite number of rigid bodies is negligible as soon as their diameters are small whereas their 
mass densities are irrelevant. In addition, hypothesis \eqref{cond} allows different bodies to shrink to zero in different order of scaling.

Let us mention that  there are some results in the context of rigid obstacles in viscous Newtonian fluids. The flow around a small rigid obstacle 
was studied by Iftimie et al. \cite{MR2244381}. Lacave \cite{MR2557320} studies the limit of a viscous fluid flow
in the exterior of a thin obstacle shrinking to a curve.

 We use the framework of \emph{weak solutions} in the spirit of Judakov \cite{MR0464811}, Gunzburger, Lee and Seregin \cite{GLSE} or Galdi \cite{GAL1}. The relevant existence theory for the fluid structure interaction problem was 
developed by San Martin, Starovoitov, and Tucsnak \cite{SST} for $d=2$ and in \cite{F3} for $d=3$. In both cases, the solutions are global--in--time and allow for possible collisions of the bodies and also collisions with the domain boundary.

Similarly to the companion paper \cite{FRZ2} concerning compressible fluids, our approach is based on a new \emph{restriction operator} that assigns a given function its ``projection'' on the space of rigid motions attached to the bodies. We point out that accommodation of several bodies needs a nontrivial modification of the construction presented 
in \cite{FRZ2}. In addition, we show new ``negative norm'' estimates of the restriction operator that are of independent interest and can be used in problems involving compressible flows.

The new restriction operator improves considerably the error estimates necessary to perform the asymptotic limit. Another new ingredient is that we use the dissipation energy rather than the energy itself to obtain suitable bounds on the rigid {body translation} velocity. This is why the result is independent of the mass densities of the bodies.

The paper is organized as follows. In Section \ref{m}, we formulate the problem and state our main result. {Next, in Section \ref{u}, we derive uniform bounds on the sequence of solutions to the fluid--structure interaction problem independent of the scaling parameter}. {In Section \ref{sec:test}, we introduce a restriction operator suitable for modifying the test function in the weak formulation of the problem.} The convergence analysis and the proof of the main result are done in Section \ref{sec:conv}.

\section{Problem formulation, main result}
\label{m}

We consider a {domain} $\Omega \subset \mathbb{R}^d$, $d=2,3$, containing a viscous, incompressible Newtonian fluid. 
Accordingly the fluid velocity $\bu$ satisfies the Navier--Stokes system of equations 
\begin{align} 
\Div \bu &= 0, \label{m1} \\ 	
\partial_t \bu + \Div (\bu \otimes \bu) + \Grad \Pi &= {\Div \mathbb{S} (\Ds \bu)} + \vc{g}, \label{m2} \\ 
{\mathbb{S}(\Ds \bu) = \mu \Ds \bu },\ 
\Ds \bu& = \frac{ \Grad \bu + \Grad^t \bu }{2} ,\ \mu > 0 \label{m3}, 
	\end{align}
where $\Pi$ is the pressure and the function $\vc{g}$ denotes an external {volume} force. 

The rigid bodies are represented by compact connected sets $\mc{S}^i 
\subset \mathbb{R}^d$, $i=1,\cdots, N$.
We suppose the bodies are immersed in the fluid and their position at a time $t > 0$ is determined by a family of 
affine isometries $(\sigma^i(t))_{t \geq 0}$, 
\[
\mc{S}^i (t) = \sigma^i [\mc{S}^i],\ \sigma^i(t)x = 
\mathbb{O}_{i} (t) x + \vc{h}_{i} (t),\ \mathbb{O}_i \in SO(d),\ t \geq 0,\ i = 1, \cdots, N.
\]
In addition, we introduce the associated rigid velocity fields, 
\begin{equation} \label{m5}
\bu_{\mc{S}^i} (t,x) = \vc{Y}_i  + \mathbb{Q}_{i} ( x - \vc{h}_i),\ 
\vc{Y}_i(t)  = \frac{\D }{\dt } \vc{h}_i (t) ,\ \mathbb{Q}_i(t) = 
\frac{ \D }{\dt} \mathbb{O}_i (t) \circ {\mathbb{O}_i}^{-1}(t).
\end{equation}

Finally, we identify the fluid region 
\[
\Omega_{f}(t) = \Omega \setminus \cup_{i = 1}^N \mc{S}^i(t),\ Q_{f} = 
\left\{ (t,x) \ \Big|\ t \in (0,T),\ x \in \Omega_{f}(t) \right\}
\]

\subsection{Weak formulation}
\label{ws}

{We suppose that the rigid bodies are immersed in the fluid. As the fluid is viscous, a natural working hypothesis asserts that both 
the velocity and the momentum coincide on the body boundary, see e.g. Galdi \cite{GAL1}.	 
Accordingly, a suitable} weak formulation of the fluid--structure interaction problem (see \cite{F3}) is based on the quantities $\left[ \rho, \bu, (\vc{h}^1 \cdots \vc{h}^N) , (\mathbb{O}^1 \cdots , \mathbb{O}^N) \right]$. In this context, it is convenient to consider both the mass density $\rho = \rho(t,x)$ and the velocity $\bu = \bu(t,x)$ as functions defined for all $x \in \R^{d}$.

\subsubsection{Regularity}

\begin{itemize} 
	
	\item The mass density $\rho$ is non--negative, 
	\begin{equation} \label{m6}
		\rho \in L^\infty ((0,T) \times \R^{d}) \cap C^1([0,T]; L^1(\Omega)).
		\end{equation}
	\item The velocity $\bu$ belongs to the Ladyzhenskaya class 
	\begin{equation} \label{m7}
		\bu \in L^\infty(0,T; L^2(\R^{d}; \R^{d})) \cap L^2(0,T; W^{1,2}(\R^{d}; \R^{d})),\ \Div \bu = 0.
		\end{equation}
	
	\item The affine isometries are Lipschitz continuous in time, 
	\begin{equation} \label{m8}
		\vc{h}_i \in W^{1, \infty}(0,T; \R^{d}),\ \mathbb{O}_i \in W^{1,\infty} (0,T; SO(d) ),\ i =1, \cdots, N.
		\end{equation}
	\end{itemize}

\subsubsection{Compatibility}

\begin{itemize}
	\item
	\begin{align} 
		\rho(t,x) &= 1 \ \mbox{for a.a.}\ x \in \Omega_f (t), \br
		\rho(t,x) &= \rho_{\mathcal{S}^i} \ \mbox{for a.a.} \ x \in \mathcal{S}^i(t),\ i = 1,\cdots,N
		 \label{m9}
		\end{align} 
	for any $t \in [0,T]$.
	\item 
	
	\begin{align} 
		\bu(t, \cdot) &\in W^{1,2}_0 (\Omega; \R^{d}) \ \label{m10} \\ 
		(\bu - \bu_{\mc{S}^i})(t, \cdot) &\in W^{1,2}_0 (\R^{d} \setminus \mc{S}^i(t); \R^{d}) \label{m11} \\
		&\mbox{for a.a.}\ t \in (0,T),\ i = 1,\cdots, N .\label{m12}
		\end{align}
	
	\end{itemize}

\subsubsection{Mass conservation} 

The equation of continuity  
\begin{equation} \label{m13}
 \int_0^T \int_{\R^{d}}{ \Big[ \rho \partial_t \varphi + \rho \bu \cdot \Grad \varphi \Big]  } \dt = - 
 \int_{\R^{d}}{ \rho_0 \varphi (0, \cdot) }	
	\end{equation}
holds for any $\varphi \in C^1_c([0,T) \times \R^{d})$. 

\subsubsection{Momentum {balance}} 

The momentum equation 
\begin{align}
\int_0^T &\int_{\R^{d}}{ \Big[ \rho \bu \cdot \partial_t \bfphi + \rho\bu \otimes \bu : \Ds \bfphi \Big] } \dt \br &= 
\int_0^T \int_{\R^{d}}{ \Big[ {\mathbb{S}( \Ds \bu ) } : \Ds \bfphi - \rho \vc{g} \cdot \bfphi \Big] } \dt - 
\int_{\R^{d}}{ \rho_0 \bu_0 \cdot \bfphi }	\label{weak:mom}
	\end{align}
holds for any function $\bfphi \in C^1_{c} ([0,T) \times \Omega; \R^{d})$, $\Div \bfphi = 0$ 
satisfying 
\begin{equation} \label{m15}
\Ds \bfphi (t, \cdot) = 0 \ \mbox{on an open neighbourhood of}\ \mc{S}^i(t) 
\ \mbox{for any}\ t \in [0,T),\ i = 1, \cdots, N.	
	\end{equation}

\subsubsection{Total energy dissipation} 

The energy inequality 
\begin{equation} \label{eq2}
	\int_{\R^{d}}{ \rho |\bu|^2 (\tau, \cdot) } + \int_0^{{\tau}} \int_{\R^{d}}{ {\mathbb{S}(\Ds \bu) : \Ds \bu } } \leq 
	\int_{\R^{d}}{ \rho_0 |\bu_0|^2  } + \int_0^{{\tau}} \int_{\R^{d}}{ \rho \vc{g} \cdot \bu } \dt
\end{equation} 
holds for a.a. $\tau \in (0,T)$. 

\subsection{Main result}

We are ready to state our main result. 


\begin{Theorem} [{\bf Asymptotic limit}] \label{Thm:main}	
	Let $\mathcal{S}^i_\ep \subset \R^d$, $i = 1, \cdots, M(\ep)$, $d=2,3$ be a family of compact {connected} sets   satisfying 
 \begin{equation}\label{M-ep} 
 M(\ep)\sim -{\alpha}\log\ep \mbox{ for some } \alpha\in (0,5/7),
 \end{equation}
\begin{align}\label{m17}
		D_\ep \equiv \max_{i=1, \cdots, N} {\rm diam}[\mc{S}^i_\ep ] \to 0 \ \mbox{as}\ \ep &\to 0, \br
	0 < \lambda D_\ep^\beta \leq |\mathcal{S}^i_\ep |  \ \mbox{as}\ \ep &\to 0,\ 
d {\leq } \beta <  \left\{ \begin{array}{l}  15-21\alpha \ \mbox{if}\  d = 3, \\ \\ 
	\mbox{arbitrary finite if}\ d = 2, \end{array} \right.		
\end{align}		
for some $\lambda > 0$ independent of $\ep$, $i=1, \cdots, N.$ Suppose the rigid body densities 
$\vr^\ep_{\mc{S}^i}$, $i =1, \cdots, N$
are constant,
\begin{equation} \label{m18}
	0 < \rho^{\ep}_{\mc{S}^i} \leq \Ov{\rho} \ \mbox{uniformly for}\ \ep \to 0, \ i = 1, \cdots N	
\end{equation}	
for some $\Ov{\vr} > 0$ independent of $\ep$.

	Let $ \left[ \rho^{\ep}, \bu^{\ep}, \left(\vc{h}_1^{\ep}, \cdots , \vc{h}_{M(\ep)}^{\ep}\right) , \left( \mathbb{O}_1^{\ep}, \cdots ,\mathbb{O}_{M(\ep)}^{\ep} \right) \right]_{\ep > 0}$ be the associated sequence of weak solutions to the fluid--structure interaction problem {specified in Section \ref{ws}}, with the initial data $(\rho^{\ep}_0, \bu^{\ep}_0)_{\ep > 0}$ satisfying 
	\begin{align} 
	\vr^{\ep}_0(x) = \left\{ \begin{array}{l} \vr^{\ep}_{\mc{S}^i} \ \mbox{if}	\ x \in \mathcal{S}^i_\ep (0),\ 
		i  = 1, \cdots, M(\ep), \\ \\
	1 \ \mbox{otherwise}, 
\end{array} \right. \label{m22} \\ 		
		\bu^{\ep}_0 \in L^2(\R^d; \R^d),\ \rho^{\ep}_0 \bu^{\ep}_0 \to 
	{\bu_0 \ \mbox{in}\ L^2(\R^d; \R^d)},  \label{m23}
		\end{align}
where 
\begin{equation} \label{m24}
	\bu_0 \in {L^2(\Omega; \R^d)},\ \Div \bu_0 = 0.
	\end{equation}
{Finally, suppose}
\begin{equation}\label{m24a} 
	\vc{g} = \vc{g}(x),\ \vc{g} \in L^2 \cap L^\infty(\R^d; \R^d). 
	\end{equation}	

Then, up to a suitable subsequence,  
\begin{align} 
	\bu^{\ep} \to \bu \ &\mbox{weakly in}\ L^2(0,T; W^{1,2}_0 (\Omega; \R^d)), \br
	&\mbox{and in}\ L^2((0,T) \times \Omega; \R^d), \label{m25} 
\end{align}
where $\bu$ is a weak solution of the Navier--Stokes system \eqref{m1}--\eqref{m3}, with the initial data 
$\bu_0$ satisfying the energy inequality 
\begin{equation} \label{m26}
	\intO{ |\bu|^2 (\tau, \cdot) } + \int_0^\tau \intO{ {\mathbb{S} (\Ds \bu) : \Ds \bu }} \leq
	\intO{  |\bu_0|^2  } + \int_0^\tau \intO{  \vc{g} \cdot \bu } \dt
\end{equation} 
for a.a. $\tau \in (0,T)$.
	\end{Theorem}


\begin{Remark} We can also tackle the case of finitely many bodies with a straightforward modification of the technique used in the previous proposition.
Let $\mathcal{S}^i_\ep \subset \R^d$, $i = 1, \cdots, N$, $d=2,3$ be a family of compact {connected} sets  satisfying 
\begin{align}\label{m17-finite}
		D_\ep \equiv \max_{i=1, \cdots, N} {\rm diam}[\mc{S}^i_\ep ] \to 0 \ \mbox{as}\ \ep &\to 0, \br
	0 < \lambda D_\ep^\beta \leq |\mathcal{S}^i_\ep |  \ \mbox{as}\ \ep &\to 0,\ 
d {\leq } \beta <  \left\{ \begin{array}{l}  15 \ \mbox{if}\  d = 3, \\ \\ 
	\mbox{arbitary finite if}\ d = 2, \end{array} \right.		
\end{align}		
for some $\lambda > 0$ independent of $\ep$, $i=1, \cdots, N.$ Suppose the rigid body densities 
$\vr^\ep_{\mc{S}^i}$, $i =1, \cdots, N$
are constant.

	Let $ \left[ \rho^{\ep}, \bu^{\ep}, (\vc{h}_i^{\ep}, \cdots , \vc{h}_N^{\ep}) , ( \mathbb{O}_1^{\ep}, \cdots ,
	\mathbb{O}_N^{\ep} ) \right]_{\ep > 0}$ be the associated sequence of weak solutions to the fluid--structure interaction problem {specified in Section \ref{ws}}, with the initial data $(\rho^{\ep}_0, \bu^{\ep}_0)_{\ep > 0}$ satisfying 
\eqref{m22}--\eqref{m23}.	
Then, we have the same conclusion for the limiting system as in \eqref{m25}--\eqref{m26}.
\end{Remark}
\begin{Remark}
In fact, hypothesis \eqref{m18} requiring uniform bounds on the body density is not restrictive. Indeed, if 
$\rho^{\ep}_{\mc{S}}$ contains an unbounded sequence, then the relevant results have been already obtained by He and Iftimie \cite{HeIft1}, \cite{HeIft2}. The main novely of the present result is allowing the body density to be asymptotically small.
\end{Remark}

\begin{Remark}
Without loss of generality, we set $D_\ep = \ep$. In accordance with hypothesis \eqref{m17}, 
there are balls $B_{r \ep}[\vc{y}_i^{\ep}]$ centred at $\vc{y}_i^{\ep} \in \R^d$ and of radius $r \ep$ such that 
\[
\mathcal{S}^i_\ep \subset B_{r \ep}[\vc{y}_i^{\ep}]. 
\]  
In addition, choosing $r > 0$ large enough, we may suppose 
\[
\vc{y}_i^{\ep} = \frac{1}{|\mathcal{S}^i_\ep |} \int_{\mathcal{S}^i_\ep} x \ \dx 
\]
coincided with the barycenter of the body $\mathcal{S}^i_\ep$. Accordingly, we may infer 
\begin{equation} \label{r1} 
\mathcal{S}^i_\ep (t) \subset B_{r \ep}[\vc{h}^{\ep}_i(t)],\ t \in [0,T], \ i=1,\cdots, N,
\end{equation}
where $(\vc{h}^{\ep}_i )_{\ep > 0}$ are the rigid translation of the barycenters of the body. Finally, again without loss of generality, we suppose $r=1$ therefore $\mathcal{S}^i_\ep (t) \subset B_{\ep}[\vc{h}^{\ep}_i(t)],\ t \in [0,T]$,\ $i = 1,\cdots, N$.
\end{Remark}
It follows from \eqref{r1} (with $r=1$) and the weak formulation of the momentum equation \eqref{weak:mom} that 
the integral identity
\begin{align}
	\int_0^T &\intRd{ \Big[ \vre \bu^\ep \cdot \partial_t \bfphi + \bu^\ep \otimes \bu^\ep : \Ds \bfphi \Big] } \dt \br &= 
	\int_0^T \intRd{ \Big[ {\mathbb{S}(\Ds \bu^\ep ) } : \Ds \bfphi - \vre \vc{g} \cdot \bfphi \Big] } \dt - 
	\intRd{ {\rho^\ep_0 \bu^\ep_0} \cdot \bfphi }	\label{r2}
\end{align}
holds for any function $\bfphi \in C^1_{c} ([0,T) \times \Omega; \R^d)$, $\Div \bfphi = 0$ 
satisfying 
\begin{equation} \label{r3}
	\Grad \bfphi (t, \cdot) = 0 \ \mbox{on}\ {B_{ \ep}}[\vc{h}^{\ep}_i(t)] 
	\ \mbox{for any}\ t \in [0,T),\ i = 1, \cdots, N.
\end{equation}
Obviously, the test functions satisfying \eqref{r3} are constant on the shifted balls containing the rigid bodies.
It is worth noting that the class of test functions \eqref{r3} is much larger than its counterpart for the 
obstacle problem, where the test functions are supposed to vanish on the obstacle.
	
\section{Uniform bounds}
\label{u} 

In this section, we {derive suitable uniform bounds necessary for passing} to the limit in the weak formulation of the momentum equation \eqref{weak:mom}. All bounds used in the limit passage follow from the energy inequality \eqref{eq2}. 

Let $\left[\rho^{\ep}, \bu^{\ep}, (\vc{h}^{\ep}_1, \cdots, \vc{h}^\ep_N), (\mathbb{O}^{\ep}_1, \cdots ,
\mathbb{O}^{\ep}_N \right]_{\ep > 0}$ be the associated sequence of weak solutions to the fluid--structure interaction problem satisfying \eqref{m6}--\eqref{eq2}. 
{As $\vc{g}$ satisfies the hypothesis \eqref{m24a}, we have} 
\[
\intRd{ \rho^\ep \vc{g} \cdot \bu^\ep } = \sum_{i =1}^N \int_{\mathcal{S}^i_{\ep}(t) } \vr^\ep_{\mathcal{S}^i} \vc{g} \cdot \bu^\ep \dx 
+ \int_{\R^d \setminus \cup_{i=1}^N \mathcal{S}_{\ep}(t)} \vc{g} \cdot \bu^\ep
\] 
{where, in accordance with the hypothesis \eqref{m18}, \eqref{m24a}:}
\[
 \int_{\mathcal{S}^i_{\ep}(t) } \vr^\ep_{\mathcal{S}^i} \vc{g} \cdot \bu^\ep \dx \aleq  \int_{\mathcal{S}^i_{\ep}(t) } \vr^\ep_\mathcal{S} \ \dx + \intRd{ \vr^\ep |\bu^\ep|^2 } \aleq 1 + \intRd{ \vr^\ep |\bu^\ep|^2 },\ i=1, \cdots, N.
\]
{Here and hereafter, the symbol $a \aleq b$ means $a \leq c b$, where $c$ is a generic constant independent of the 
scaling parameter $\ep$. Similarly, by virtue of hypothesis \eqref{m24a} and Cauchy--Schwartz inequality, }
\[
\int_{\R^d \setminus \cup_{i=1}^N\mathcal{S}^i_{\ep}(t)} \vc{g} \cdot \bu^\ep \leq c(\delta) + \delta \intRd{ |\bu^\ep|^2 }  
\]
{for any $\delta > 0$. In addition, using Korn--Poincar\' e inequality, we get}
\begin{equation} \label{u1a}
\intRd{ |\bu^\ep|^2 } \aleq \intRd{ \mathbb{S} (\Ds \bu^\ep) : \Ds \bu^\ep } + 
\intRd{ \rho^\ep |\bu^\ep |^2 }.
\end{equation}
{Finally, we apply Gronwall's argument to the energy inequality \eqref{eq2} and deduce 
	the following bounds}
\begin{equation} \label{u1}  
	{\rm ess} \sup_{t \in (0,T)} \intRd{ \rho^{\ep}|\bu^{\ep}|^2 (t, \cdot) } \aleq 1,\ 
	\int_0^T \intRd{ \mathbb{S} (\Ds \bu^\ep) : \Ds \bu^\ep } \aleq 1, 
\end{equation}
{which, together with \eqref{u1a}, yields}
\begin{equation} \label{u2}	
	\int_0^T \intRd{ |\bu^{\ep}|^2 + |\Grad \bu^{\ep}|^2 } \dt \aleq 1. 
\end{equation}

\subsection{Bounds on the rigid velocity $d=2$}

In view of \eqref{m11}, $\bu^\ep(t, \cdot) = \bu^{\ep}_{\mc{S}^i}(t, \cdot)$ a.a. on $\mathcal{S}^i_\ep(t)$ for a.a. $t \in (0,T)$, $i=1,\cdots, N$. {As the rigid body densities are constant}, the translational and rotational velocities are orthogonal on 
{$\mathcal{S}^i_\ep (t)$}, specifically 
\[
\int_{{\mathcal{S}^i_\ep (t)}} \vc{Y}^{\ep}_i (t) \cdot \mathbb{Q}^{\ep}_i ( \cdot - \vc{h}^{\ep}_i (t))  \ \dx = 0.
\]
Consequently,  we have
\begin{equation} \label{u3}
\int_{\mathcal{S}^i_\ep (t) } |\vue|^2 \dx	=
\int_{\mathcal{S}^i_\ep (t) } |\bu^{\ep}_{\mc{S}^{i}}|^2 \dx = 
\int_{\mathcal{S}^i_\ep (t) } |\vc{Y}^{\ep}_i (t) + \mathbb{Q}^{\ep}_i ( \cdot - \vc{h}^{\ep}_i (t)) |^2 \geq  
\int_{\mathcal{S}^i_\ep (t) } |\vc{Y}^{\ep}_i (t)|^2 \ \dx
\end{equation} 

{If $d=2$}, the standard Sobolev embedding relation yields 
\[
\| \vue \|_{L^q (\mathcal{S}^i_\ep (t); \R^2 )} \leq c(q) \| \vue \|_{W^{1,2}(\R^2; \R^2)} 
\ \mbox{for any finite}\ 1 \leq q < \infty,
\]
where the constant $c(q)$ is \emph{independent} of $\ep$, $i=1,\cdots, N$.

Next, by \eqref{u3} and interpolation, 
\[
| \mathcal{S}^i_\ep |^{\frac{1}{2}} | \vc{Y}^{\ep}_i | \leq 
\| \vue \|_{L^2 (\mathcal{S}^i_\ep (t); \R^2 )} \leq  \| \vue \|_{L^q (\mathcal{S}^i_\ep (t); \R^2 )}
| \mathcal{S}^i_\ep |^{\frac{1}{2} - \frac{1}{q}}
\leq c(q) \| \vue \|_{W^{1,2}(\R^2; \R^2)} | \mathcal{S}^i_\ep |^{\frac{1}{2} - \frac{1}{q}}.
\]

Consequently, in view of the bound \eqref{u2}, we may infer 
\begin{equation} \label{u52}
\| \vc{Y}^\ep_i \|_{L^2(0,T)} \aleq |\mathcal{S}^i_\ep |^{- \frac{1}{q}} \ \mbox{for any finite}\ 
1 \leq q < \infty, \ i = 1,\cdots, N, \ d = 2.	
	\end{equation}

\subsection{Bounds on the rigid velocity $d=3$}

Now, we repeat the arguments of the previous section with $q = 6$ obtaining 
\begin{equation} \label{u53}
	\| \vc{Y}^\ep_i \|_{L^2(0,T)} \aleq |\mathcal{S}^i_\ep |^{- \frac{1}{6}},\ i = 1,\cdots, N, \ d = 3.
\end{equation}


\section{Restriction operator}\label{sec:test}

{A suitable choice of the restriction operator is crucial in our analysis. In contrast to the 
	overwhelming amount of the available literature, where the test functions are modified to vanish on 
	the body, we take advantage of the freedom allowed by  \eqref{r2}, \eqref{r3} and 
{replace the function on a ball of radius $\varepsilon$ by its integral average over that ball}. The same idea has already been exploited in \cite{FRZ2}, where detailed proofs of the statements collected below are available}. 

Consider a function \begin{align} 
H &\in C^\infty(\mathbb{R}), \ 0 \leq H(Z) \leq 1,\ H'(Z) = H'(1- Z) \ \mbox{for all}\ Z \in \mathbb{R}, \br
H(Z) &= 0 \ \mbox{for} \ - \infty < Z \leq \frac{1}{4},\ 
H(Z) = 1 \ \mbox{for}\ \frac{3}{4} \leq Z < \infty
\nonumber
\end{align}

{For $\varphi \in {L^1_{\rm loc}(\R^d)}$ , $r > 0$ we define  $E_r$,}
\begin{align}\label{def:E}
E_r [\varphi] (x) = 
\frac{1}{|B_r  |} \int_{B_r} \varphi \ dz \ H \left( 2 - \frac{|x|}{r} \right)
+ \varphi(x) H \left( \frac{ | x |}{r} - 1 \right).
\end{align}
{where $B_r $ denotes the ball centred at zero with the radius $r > 0$.} Obviously, the operator 
$E_r$ maps the space $\DC(\R^d)$ into itself, and, moreover 
\begin{align} 
	\| E_r [\varphi] \|_{L^p(\R^d)} \aleq \| \varphi \|_{L^p(\R^d)}, 
	\label{T3} \\ 
\| \Grad E_r [\varphi] \|_{L^p(\R^d; \R^d)} \aleq \| \Grad \varphi \|_{L^p(\R^d; \R^d)}
\ \mbox{for any}\ 1 \leq p \leq \infty,	
\label{T4}
	\end{align}
uniformly for $0 < r \leq 1$, see \cite[Section 4.1]{FRZ2}.

\subsection{{Restriction operator in the class of solenoidal functions}}

\label{RO}

{The operator $E_r$ does not preserve solenoidality if applied componentwise to a solenoidal function. To fix this problem, we introduce 
the operator}
\begin{equation} \label{E6}
	{\vc{R}_r} [\bphi ] = E_r [\bphi] - \mathcal{B}_{2 r, r} \Big[ 
	\Div E_r [\bphi]|_{B_{2r} \setminus {B}_{r}} \Big],
	\end{equation}
where $\mathcal{B}_{2 r,  r}$ is a suitable branch of the inverse of the divergence operator defined on the 
annulus $B_{2r} \setminus B_{r}$. A possible construction of $\mathcal{B}$ was proposed by Bogovskii 
\cite{BOG} and later elaborated by Galdi \cite{MR2808162} followed by 
Diening et al. \cite{DieRuzSch} , Gei{\ss}ert \cite{GEHEHI} et al. among others. In our setting, the operator 
$\mathcal{B}_{2r,r}$ can be constructed by a simple scaling argument:

\begin{itemize}

\item 

We start with the annulus 
\[
O = B_2 \setminus B_1 = \left\{ y \in \R^d \ \Big|\ 1 < |y| < 2 \right\}.
\]
Following Galdi \cite[Chapter III, Section III.3]{MR2808162}, Diening et al. \cite{DieRuzSch}, 
we construct a linear operator $\mathcal{B}_O$ defined 
	{\it a priori} on smooth functions $g \in \DC (O)$, $\int_O g \ \D y= 0$, enjoying the following properties:

\item 
\[
\mathcal{B}_O [g] \in \DC(O; \R^d),\ {\rm div}_y \mathcal{B}_O [g] = g ;
\]
\item 
\begin{equation} \label{E6a}
\| \nabla_y \mathcal{B}_O [g] \|_{L^p(O; \R^{d \times d})} \leq c_I(p, O) \| g \|_{L^p(O)}
\ \mbox{for any}\ 1 < p < \infty.
\end{equation}
Thanks to this property, $\mathcal{B}_O [g]$ can be extended as a bounded linear operator on the space 
\[
L^p_0 (O) = \left\{ g \in L^p(O)\ \Big|\ \int_O g \ \D y = 0 \right\} 
\]
ranging in $W^{1,p}_0(O, \R^d)$.

\item If, in addition, the function $g$ can be written as $g = {\rm div}_y \vc{f}$, where 
$\vc{f} \in L^q(O; \R^d)$ satisfies 
\begin{equation} \label{norm}
	\vc{f} \cdot \vc{n}|_{\partial O} = 0,
	\end{equation}
then 
\begin{equation} \label{E6bb}
\| \mathcal{B}_O [ {\rm div}_y \vc{f}] \|_{L^q(O; \R^{d})} \leq c_{II}(q, O) \| \vc{f} \|_{L^q(O; \R^d)},\ 
1 < q < \infty.
\end{equation}

\item If  $g \in W^{l,p}_0$, then
\[
\| \mathcal{B}_O [g] \|_{W^{l+1,p}_0(O, \R^d)} \leq c(l,p,O) \| g \|_{W^{l,p}_0(O; \R^d)},\ l = 0,1,\cdots, \ 1 < p < \infty,
\]
see Galdi \cite[Theorem III.3.3]{MR2808162}.

\item An analogue of $\mathcal{B}_O$ on the domain $B_{2r} \setminus B_r$ can be now defined via scaling. 
To a given function $g$ defined on $B_{2r} \setminus B_r$ and satisfying $\int_{B_{2r} \setminus B_r} g \dx = 0$, 
we associate a function $\widetilde{g}$ on $O$, 
\[
\widetilde{g}(y) = g (ry),\ y \in O.
\]
We set 
\[
\mathcal{B}_{2r,r}[g] (x) = r \mathcal{B}_O [\widetilde{g}] \left( \frac{x}{r} \right),\ x \in B_{2r} \setminus 
B_r.
\] 
Seeing that 
\[
\Div \mathcal{B}_{2r,r}[g] (x) = {\rm div}_y \mathcal{B}_O [\widetilde{g}] \left( \frac{x}{r} \right), 
\Grad \mathcal{B}_{2r,r}[g] (x) = \nabla_y \mathcal{B}_O [\widetilde{g}] \left( \frac{x}{r} \right), 
\]
and
\[
\mathcal{B}_{2r,r}[\Div \vc{f}] (x) = \mathcal{B}_O [ {\rm div}_y \widetilde{\vc{f}}] \left( \frac{x}{r} \right)
\]
we easily observe that $\mathcal{B}_{2r,r}$ shares all properties of $\mathcal{B}_O$ on the domain 
$B_{2r} \setminus B_r$. Moreover, the bounds \eqref{E6a}, \eqref{E6bb} 
are satisfied with the same
with the same constants $c_I(p,O)$, $c_{II}(p,O)$:  
\begin{align}
	\| \Grad \mathcal{B}_{2 r, r} [g] \|_{L^p( {B_{2 r} \setminus B_r }  ; \R^{d \times d})} \leq c_I(p) \| g \|_{L^p({B_{2 r} \setminus B_r })}
	\ &\mbox{whenever}\ \int_{B_{2 r} \setminus B_r } g \dx = 0 \br 
	 &\mbox{for any}\ 1 < p < \infty,  \label{E6b}
\end{align}
and 
\begin{align} 
\| \mathcal{B}_{2 r, r} [\Div \vc{f}] \|_{L^q( B_{2 r} \setminus B_r; \R^{d})} &\leq c_{II}(q) \| \vc{f} \|_{L^q(B_{2 r} \setminus B_r; \R^d)},\br 
\mbox{whenever}\ \vc{f} \cdot \vc{n}|_{\partial (B_{2 r \setminus B_r}} = 0 \ \mbox{for any}\ 
1 < q < \infty.	
\label{E6cc}	
	\end{align}
with the constants $c_I$, $c_{II}$ independent of $r > 0$.

	\end{itemize}

  Recalling the definition of the operator $E_r$ as in \eqref{def:E}, we have 
\begin{equation} \label{E100}
E_{r}[\bfphi] = \mbox{constant vector in} \ B_{\frac{5}{4}r} \setminus B_r ,\ 
E_{r}[\bfphi] = \bfphi \ \mbox{in}\ B_{2 r} \setminus B_{\frac{7}{4} r}.
\end{equation}
In particular, 
\[
\int_{B_{2r} \setminus B_r } \Div E_r[\bfphi] \dx = 0 
\ \mbox{whenever}\ \Div \bfphi = 0.
\]
Therefore, going back to formula \eqref{E6}, $\vc{R}_r$ is well--defined for solenoidal functions. Moreover, obviously, 
\[
\Div \vc{R}_r [\bfphi ] = \Div \bfphi.
\]
Using the uniform bound \eqref{E6b} together with \eqref{T4} we deduce 
\begin{equation} \label{E101}
\left\| \Grad \vc{R}_r [\bfphi ] \right\|_{L^p(\R^d, \R^{d \times d})} \aleq \left\| \Grad \bfphi  \right\|_{L^p(\R^d, \R^{d \times d})},\ 1 < p < \infty
	\end{equation}
uniformly for $0 < r \leq 1$.

Finally, we claim the $L^q-$bound 
\begin{equation} \label{E102}
	\left\|  \vc{R}_r [\bfphi ] \right\|_{L^q(\R^d, \R^{d})} \aleq \left\| \bfphi  \right\|_{L^q(\R^d; \R^d)},\ 1 < q < \infty
\end{equation}
uniformly for $0 < r \leq 1$.
Note that this \emph{is not} a direct consequence of the ``negative'' estimates stated in \eqref{E6cc} as 
$E_r [\bfphi] \cdot \vc{n}$ may not vanish on $\partial (B_{2 r} \setminus B_r)$. To see \eqref{E102}, we first 
construct a family of cut--off functions 
\[
\psi_r \in \DC (B_{2r} \setminus B_r),\ 
0 \leq \psi_r \leq 1,\ \psi_r(x) = 1 \ \mbox{for all}\  \frac{5}{4} r \leq |x| \leq \frac{7}{4} r,\
|\Grad \psi_r | \aleq \frac{1}{r}.
\]
In accordance with \eqref{E100}, we have 
\[
\Div E_r [\bfphi] = \psi_r \Div E_r [\bfphi] 
\ \mbox{provided}\ \Div \bfphi = 0.
\]
Accordingly, 
\[
\mathcal{B}_{2 r, r} [\Div E_r [\bfphi]] = 
\mathcal{B}_{2 r, r} [\psi_r \Div E_r [\bfphi]] = 
\mathcal{B}_{2 r, r} [\Div (\psi_r E_r [\bfphi])] - \mathcal{B}_{2 r, r}[ \Grad \psi_r \cdot 
E_r [\bfphi]].
\]
On the one hand, as $\psi_r \in \DC (B_{2r} \setminus B_r)$, we are allowed to apply the negative bound \eqref{E6cc} to obtain 
\begin{equation} \label{E103}
\left\| \mathcal{B}_{2 r, r} [\Div (\psi_r E_r [\bfphi])] \right\|_{L^q(\R^d; \R^d)} 
\aleq \left\| \psi_r E_r [\bfphi] \right\|_{L^q(\R^d; \R^d)} \aleq 	\left\| \bfphi \right\|_{L^q(\R^d; \R^d)}.
	\end{equation}
On the other hand, the gradient bounds \eqref{E6b} yield 
\begin{equation} \label{E104}
\left\| \Grad \mathcal{B}_{2 r, r}[ \Grad \psi_r \cdot 
E_r [\bfphi]]	\right\|_{L^q (\R^d; \R^{d \times d})} \aleq 
\left\| \Grad \psi_r \cdot 
E_r [\bfphi]]	\right\|_{L^q (\R^d)} \aleq \frac{1}{r} \left\| \bfphi \right\|_{L^q(\R^d; \R^d)}.
	\end{equation}
Next, by virtue of Poincar\' e inequality, 
\begin{equation} \label{E105}
\left\| \mathcal{B}_{2 r, r}[ \Grad \psi_r \cdot 
E_r [\bfphi]]	\right\|_{L^q (\R^d; \R^{d})} \aleq r \left\| \Grad \mathcal{B}_{2 r, r}[ \Grad \psi_r \cdot 
E_\ep [\bfphi]]	\right\|_{L^q (\R^d; \R^{d \times d})}.
\end{equation}
Combining the estimates \eqref{E103}--\eqref{E105} we obtain 
\[
\left\| \mathcal{B}_{2 r, r} [\Div E_r [\bfphi]] \right\|_{L^q(\R^d; \R^d)} \aleq \left\| \bfphi \right\|_{L^q(\R^d; \R^d)},
\]
which, together with \eqref{T3} yields the desired conclusion \eqref{E102}.

\subsection{Space shift}
	
{For $\vc{h} \in \R^d$,  we set}	 
\begin{equation} \label{E10}
	\vc{R}_r (\vc{h} )[\bphi] = S_{- \vc{h}} \vc{R}_r \Big[ S_{\vc{h}} [\bphi] \Big],
	\end{equation}
where $S_{\vc{h}}$ is the shift operator given by 
\[
S_{\vc{h}} [f] (x) = f ( \vc{h} + x ).
\]

{The basic properties of the operator $\vc{R}_r(\vc{h})$ are summarized below, 
	cf. also \cite[Proposition 5.1] {FRZ2}. }

\begin{Proposition} \label{EP1}
	
	The operator $\vc{R}_r (\vc{h})$ is well defined for any function $\bfphi$ in the class  
	\[
	\bfphi  \in \DC(\R^d; \R^d), \ \Div \bfphi = 0
	\]
and can be uniquely extended to functions 
\[
\bfphi \in L^p_{\rm loc}(\R^d; \R^d),\ \Div \bfphi = 0 \ a.a.
\]	

Moreover, the following holds:

\begin{itemize}
	
	\item
	
\begin{equation} \label{reg1}
	\bfphi \in C^\infty (\R^d; \R^d) \ \Rightarrow \ \vc{R}_r (\vc{h})[\bfphi] \in C^\infty(\R^d; \R^d);  
\end{equation}	

\item 
\begin{equation} \label{E13}
	\Div \vc{R}_r (\vc{h})[\bphi] = \Div \bphi = 0 ;
\end{equation}	
	 
	\item 
\begin{equation} \label{E12}
\vc{R}_r (\vc{h})[\bphi] = \left\{ \begin{array}{l} \frac{1}{|{B}_{r}(\vc{h})|} 
\int_{{B}_{r}(\vc{h})} \bphi \dx \ \mbox{if}\ |x - \vc{h}| < r,  \\ 
\\ 
\bphi (x) \ \mbox{if}\ |x - \vc{h}| > 2 r;  \end{array}	\right. 
	\end{equation} 

\item 

\begin{equation} \label{E12a}
\vc{R}_r (\vc{h})	[\bfphi] = \bfphi \ \mbox{whenever}\ 
\bfphi \ \mbox{is a constant vector on}\ B_{2r}(\vc{h});
	\end{equation}

\item

\begin{equation} \label{E14bis}
	\left\| \vc{R}_r (\vc{h})[\bphi] \right\|_{L^p(\R^d; \R^{d})} \aleq 
	\| \bphi\|_{L^{p}(\R^d; \R^{d})}
\end{equation}
\begin{equation} \label{E14}
\left\| \Grad \vc{R}_r (\vc{h})[\bphi] \right\|_{L^p(\R^d; \R^{d \times d})} \aleq 
\| \Grad \bphi\|_{L^{p}(\R^d; \R^{d\times d})}
\end{equation} 
for any $1 < p < \infty$ independently of $0 < r \leq 1$;
\item  If $\bphi$ is compactly supported, then so is {$\vc{R}_r (\vc{h})[\bphi]$}. 
Specifically, 
 \begin{equation} \label{E15}
{\rm supp}[ \vc{R}_r (\vc{h})[\bphi] ] \subset \Ov{ \mathcal{U}_{2r}[ {\rm supp}[\bfphi ]] }
\end{equation}
where $\mathcal{U}_{2r}(O)$ denotes the $2r-$neighbourhood of a set $O$.
\end{itemize}
\end{Proposition}

Finally, we evaluate the differential of $\nabla_{\vc{h}} \vc{R}_r (\vc{h})[\bfphi]$ for 
	a given function $\bfphi$. It follows from \eqref{reg1} that 
$\vc{R}_h (\vc{h})[\bfphi]$	is a smooth function of $x$ as long as $\bfphi$ is smooth and 
the differentiation can be performed in a direct manner. Consequently, we obtain 	
\begin{equation} \label{formula}
\nabla_{\vc{h}} \vc{R_r}(\vc{h}) [\bfphi ] = \Grad \left( \vc{R}_r (\vc{h}) [\bfphi] \right) - 
\vc{R}_r (\vc{h}) [ \Grad \bfphi ] .	
	\end{equation}
Note carefully that if $\bfphi$ is solenoidal, meaning $\Div\bfphi = 0$, then so is 
$\Grad \bfphi$ (component--wise) and the right--hand side of \eqref{formula} is well defined.
Using the bounds \eqref{E14bis}, \eqref{E14}, the formula \eqref{formula} can be extended to solenoidal functions 
$\bfphi \in W^{1,p}_{\rm loc}(\R^d; \R^d)$ by density argument. 

\subsection{Composition}

As we are facing the $N$-body problem, it is convenient to consider the composition of $N$ restriction operators
\begin{equation} \label{CC1}
\vc{R}_\ep (\vc{h}_1, \cdots \vc{h}_N) [\bfphi]  
= \vc{R}_\ep (\vc{h}_1 ) \circ \vc{R}_{5\ep} (\vc{h}_2) \circ \cdots \circ \vc{R}_{5^{N-1}\ep} (\vc{h}_N) [\bfphi]	
	\end{equation}
for arbitrary $(\vc{h}_1, \cdots, \vc{h}_N)$. In view of the property \eqref{E13}, the operator 
$\vc{R}_\ep (\vc{h}_1, \cdots \vc{h}_N) [\bfphi]$ is well defined for any solenoidal $\bfphi$ and 
\[
\Div \vc{R}_\ep (\vc{h}_1, \cdots \vc{h}_N) [\bfphi] = 0.
\] 

The following result is crucial. 

\begin{Lemma} \label{LB1}
	Suppose $\vc{h}_1, \cdots, \vc{h}_N$ are $N$ points in $\R^d$, $N \geq 1$. Let 
	$r_1, \cdots, r_N$ be positive numbers, 
	\begin{equation} \label{B1}
		r_{n+1} \geq 5 r_n,\ n = 1,2,\cdots, N-1.
	\end{equation} 
	
	Then 
	\begin{equation} \label{B2}
		\vc{R}_{r_1}(\vc{h}_1) \circ \cdots \circ \vc{R}_{r_N} ( \vc{h}_N ) [\bfphi] = 
		\Lambda_i \ \mbox{on}\ B_{r_1}(\vc{h}_i),\ i =1,\cdots, N, 
	\end{equation}
	where $\Lambda_i$ are constant vectors.

\end{Lemma}

\begin{proof}

The proof can be done by induction with respect to $N$. In view of \eqref{E12}, the result obviously holds for $N=1$. 

Suppose we have already shown the conclusion for $N$ points and consider   
\[
\vc{w} = \vc{R}_{r_1}(\vc{h}_1) \circ \cdot \circ \vc{R}_{r_{N}} ( \vc{h}_{N} ) \circ \vc {R}_{r_{N + 1}}( \vc{h}_{N + 1}) [\bfphi].	
\]
In view of the induction hypothesis, it is enough to show 
\begin{equation} \label{B3}
	\vc{w} = \Lambda_{N+1} \ \mbox{on}\ B_{r_1}(\vc{h}_{N+1}). 
\end{equation}	
Denote 
\[
\vc{v} = \vc{R}_{r_2}(\vc{h}_2) \circ \cdots \circ \vc{R}_{r_{N}} ( \vc{h}_{N} ) \circ \vc{R}_{r_{N + 1}} ( \vc{h}_{N + 1} ) [\bfphi]
\]
Using again the induction hypothesis we get 
\begin{equation} \label{B4}
	\vc{v} = \Lambda_{N+1} \ \mbox{on}\ B_{r_2}(\vc{h}_{N+1}).
\end{equation}

Now, consider two complementary cases 
\[
\vc{h}_{N+1} \in \R^d \setminus B_{3r_1}(\vc{h}_1) 
\ \mbox{or}\ \vc{h}_{N+1} \in B_{3r_1}(\vc{h}_1).
\]
If $\vc{h}_{N+1} \in \R^d \setminus B_{3r_1}(\vc{h}_1)$, then it follows from \eqref{E12}
\[
\vc{w} = \vc{R}_{r_1}(\vc{h}_1)[\vc{v} ] = \vc{v} \ \mbox{on}\ B_{r_1}(\vc{h}_{N+1})
\]
and the desired conclusion follows from \eqref{B4} as $r_2 > r_1$.

If $\vc{h}_{N+1} \in B_{3r_1}(\vc{h}_1)$, then it follows from \eqref{B4} and the hypothesis $r_2 \geq 5 r_1$ that 
$\vc{v}$ equals $\Lambda_{N+1}$ on $B_{2r_1}(\vc{h}_1)$. Consequently, by virtue of \eqref{E12a},  
\[
\vc{w} = \vc{R}_{r_1}(\vc{h}_1)[\vc{v}] = \vc{v} 
\]
and the desired result follows again from \eqref{B4}.

\end{proof}

Let us summarize the properties of the operator $\vc{R}_\ep (\vc{h}_1, \cdots, \vc{h}_N)$ that can be easily deduced from Proposition \ref{EP1} and Lemma \ref{LB1}.

\begin{Proposition} \label{EP1R}
	
	The operator $\vc{R}_\ep (\vc{h}_1, \cdots,\vc{h}_N)$ is well defined for any function $\bfphi$ in the class  
	\[
	\bfphi  \in \DC(\R^d; \R^d), \ \Div \bfphi = 0
	\]
	and can be uniquely extended to functions 
	\[
	\bfphi \in L^p_{\rm loc}(\R^d; \R^d),\ \Div \bfphi = 0 \ a.a.
	\]	
	
	Moreover, the following holds:
	
	\begin{itemize}
		
		\item
		
		\begin{equation} \label{reg1R}
			\bfphi \in C^\infty (\R^d; \R^d) \ \Rightarrow \ \vc{R}_\ep  (\vc{h}_1, \cdots,\vc{h}_N) \in C^\infty(\R^d; \R^d);  
		\end{equation}	
		
		\item 
		\begin{equation} \label{E13R}
			\Div \vc{R}_\ep  (\vc{h}_1, \cdots,\vc{h}_N) [\bphi] = \Div \bphi = 0 ;
		\end{equation}	
		
		\item 
		\begin{equation} \label{E12R}
			\vc{R}_\ep (\vc{h}_1, \cdots,\vc{h}_N) [\bphi] = \Lambda_i - \mbox{a constant vector on} 
		\ B_\ep (\vc{h}_i),\ i = 1,\cdots, N;	
		\end{equation}

		\item 
		
		\begin{equation} \label{E12RR}
			\vc{R}_\ep (\vc{h}_1, \cdots,\vc{h}_N) [\bphi] (x) = \bfphi (x) \ \mbox{whenever}\ 
			x \in \R^d \setminus \cup_{i=1}^N B_{2 \cdot 5^{i-1} \ep} (\vc{h}_i) ;
		\end{equation}
		
		\item
		\begin{equation} \label{E14bisR}
			\left\| \vc{R}_\ep (\vc{h}_1, \cdots,\vc{h}_N)[\bphi] \right\|_{L^p(\R^d; \R^{d})} \leq c(p)^{N}
			\| \bphi\|_{L^{p}(\R^d; \R^{d})}
		\end{equation}
		\begin{equation} \label{E14R}
			\left\| \Grad \vc{R}_\ep (\vc{h}_1, \cdots,\vc{h}_N)[\bphi] \right\|_{L^p(\R^d; \R^{d \times d})} \leq
			c(p)^{N} 
			\| \Grad \bphi\|_{L^{p}(\R^d; \R^{d\times d})}
		\end{equation} 
		for any $1 < p < \infty$ uniformly for $0 < \ep \leq 1$;
		\item  If $\bphi$ is compactly supported, then so is {$\vc{R}_\ep (\vc{h}_1, \cdots,\vc{h}_N)[\bphi]$}. 
		Specifically, 
		\begin{equation} \label{E15R}
			{\rm supp}[ \vc{R}_\ep (\vc{h}_1, \cdots, \vc{h}_N)[\bphi] ] \subset \Ov{ \mathcal{U}_{2\cdot 5^N \ep}[ {\rm supp}[\bfphi ]] }.
		\end{equation}

	\end{itemize}
\end{Proposition}

Similarly to the preceding part, we may compute the gradients with respect to the parameters 
$\vc{h}_i$, $i=1,\cdots, N$. A straightforward application of formula \eqref{formula} yields:
\begin{align} 
\nabla_{\vc{h}_i} &\vc{R}_\ep(\vc{h}_1, \cdots, \vc{h}_N) [\bfphi] \br &= 
\vc{R}_\ep (\vc{h}_1, \cdots, \vc{h}_{i-1}) \Big[ \Grad \vc{R}_{5^{i-1} \ep}(\vc{h}_i, \cdots, \vc{h}_N) [\bfphi]
- \vc{R}_{5^{i-1} \ep}(\vc{h}_i) \left[ \Grad \vc{R}_{5^i \ep}(\vc{h}_{i+1}, \cdots, \vc{h}_N) [\bfphi]                           \right]  \Big] \br 	
&= 
\vc{R}_\ep (\vc{h}_1, \cdots, \vc{h}_{i-1}) \Big[ \Grad \vc{R}_{5^{i-1} \ep}(\vc{h}_i, \cdots, \vc{h}_N) [\bfphi] \Big] \br  &
\quad - \vc{R}_\ep (\vc{h}_1, \cdots, \vc{h}_{i}) \Big[ \Grad \vc{R}_{5^i \ep}(\vc{h}_{i+1}, \cdots, \vc{h}_N) [\bfphi]                                                 \Big],\ i = 1,\cdots, N,
\label{chain}
	\end{align}
with the convention  
\[
\vc{R}_\ep (\vc{h}_1, \cdots, \vc{h}_{0}) = 
\vc{R}_{5^N \ep} (\vc{h}_{N+1}, \cdots, \vc{h}_N) = {\rm Id}.
\]

\section{Convergence: Proof of Theorem \ref{Thm:main}}

\label{sec:conv}

Let  $\bfphi \in \DC ([0,T) \times \Omega; \R^d)$, $\Div \bfphi = 0$ be a smooth solenoidal function. 
Our ultimate goal is to tackle the infinitely many bodies case. To do so, we want to plug $\vc{R}_\ep ( \vc{h}^\ep_1, \cdots, \vc{h}^\ep_{M(\ep)})[ \bfphi ]$ as a test function in the ``relaxed'' momentum balance \eqref{r2} and perform the limit $\ep \to 0$ (in which case we also have $M(\ep)\to\infty$). 
It follows from Proposition \ref{EP1R}, notably \eqref{E13R}, \eqref{E12R} and \eqref{E15R}, that 
$\vc{R}_\ep ( \vc{h}^\ep_1, \cdots, \vc{h}^\ep_{M(\ep)})[ \bfphi ]$ belongs to the class \eqref{r3} and therefore 
represent an eligible test function for \eqref{r2} as long as we check the regularity of its time derivative. 
This will be done in the next section. 

\subsection{Error estimates for test functions}

Given a test function  $\bfphi \in \DC ([0,T) \times \Omega; \R^d)$, $\Div \bfphi = 0$, the time derivative 
of its approximation $\vc{R}_\ep (\vc{h}^\ep_1, \cdots, \vc{h}^\ep_{M(\ep)} )[\bfphi]$
can be computed directly from formula \eqref{chain}:
\begin{align} 
	\partial_t &\vc{R}_\ep (\vc{h}^\ep_1, \cdots, \vc{h}^\ep_{M(\ep)} )[\bfphi] = 
\vc{R}_\ep (\vc{h}^\ep_1, \cdots, \vc{h}^\ep_{M(\ep)} )[\partial_t \bfphi] + 
\sum_{i=1}^{M(\ep)} \nabla_{\vc{h}_i} \vc{R}_\ep (\vc{h}^\ep_1, \cdots, \vc{h}^\ep_{M(\ep)} )[\bfphi] \cdot \vc{Y}^\ep_i \br 
&= \vc{R}_\ep (\vc{h}^\ep_1, \cdots, \vc{h}^\ep_{M(\ep)} )[\partial_t \bfphi] \br 
&+ \sum_{i=1}^{M(\ep)} \vc{R}_\ep (\vc{h}_1^\ep, \cdots, \vc{h}_{i-1}^\ep) \Big[ \Grad \vc{R}_{5^{i-1} \ep}(\vc{h}_i^\ep, \cdots, \vc{h}_{M(\ep)}^\ep) [\bfphi] \cdot \vc{Y}^\ep_i \br &
- \vc{R}_{5^{i-1} \ep}(\vc{h}_i^\ep) \left[ \Grad \vc{R}_{5^i \ep}(\vc{h}_{i+1}^\ep, \cdots, \vc{h}_{M(\ep)}^\ep) [\bfphi] \cdot \vc{Y}^\ep_i                           \right]  \Big]
 \label{TD}
\end{align}
As a matter of fact, the functions $\vc{h}^\ep$ are merely Lipschitz; whence the identity for time derivatives 
\[
\frac{\D }{\dt} \vc{h}^\ep_i (t) = \vc{Y}^\ep_i 
\ \mbox{holds for only for a.a.}\ t \in (0,T).
\]
Still formula \eqref{TD} as well as eligibility  of $\vc{R}_\ep (\vc{h}^\ep_1, \cdots, \vc{h}^\ep_{M(\ep)} )[\bfphi]$
as a test function in \eqref{r2} can be verified by a density argument.

To derive the error estimates on the difference 
\[
\bfphi - \vc{R}_\ep ( \vc{h}^\ep_1, \cdots, \vc{h}^\ep_{M(\ep)} )[ \bfphi ] ,\ \bfphi \in \DC([0, T) \times \Omega; \R^d),\ \Div \bfphi = 0
\]
we use essentially two facts: 
\begin{itemize}
	
	\item For any fixed $t \in [0,T)$:
	\begin{equation} \label{prop1}
	\bfphi (t, \cdot)  - \vc{R}_\ep ( \vc{h}^\ep_1 , \cdots, \vc{h}^\ep_{M(\ep)} )[ \bfphi ] (t, \cdot) = 0 \ \mbox{outside the union of balls}\ \cup_{i=1}^{M(\ep)} B_{10^{M(\ep)} \ep}[\vc{h}^\ep_i (t) ],
	\end{equation}
which far from being optimally stated consequence of \eqref{E12RR};

	\item $\bfphi$ is smooth, in particular Lipschitz in $[0,T] \times \Omega$.
	
	\end{itemize}
In view of the bounds \eqref{E14bisR}, \eqref{E14R}, we get 
\begin{equation}\label{Rphi}
\| \vc{R}_\ep ( \vc{h}^\ep_1, \cdots, \vc{h}^\ep_{M(\ep)} )[ \bfphi ] \|_{W^{1,p}(\R^d)} \leq c(p)^{M(\ep)} \| \bfphi \|_{W^{1,p}(\R^d)}.
\end{equation}
By virtue of \eqref{prop1}, it is enough to estimate
\begin{equation*}
	\| ( \bfphi - \vc{R}_\ep ( \vc{h}^\ep_1, \cdots, \vc{h}^\ep_{M(\ep)} )[ \bfphi ] )(t, \cdot) \|_{W^{1,p}(\cup_{i=1}^{M(\ep)} B_{10^{M(\ep)} \ep}[\vc{h}^\ep_i (t) ]; \R^d)} .
\end{equation*}
 If we take $\max\{c(p)^{M(\ep)},10^{M(\ep)}\} \sim \ep^{-\alpha}$, i.e, $M(\ep)\sim -{\alpha}\log\ep$ for some $\alpha\in (0,5/7)$, then we can estimate:
\begin{equation*}
\|\bfphi\|^{p}_{W^{1,p}(\cup_{i=1}^{M(\ep)}B_{10^{M(\ep)} \ep}[\vc{h}^\ep_i (t) ]; \R^d)} \leq \|\bfphi\|^{p}_{C^{1}(\Omega)}M(\ep)|B_{10^{M(\ep)}\ep}|\leq \|\bfphi\|^{p}_{C^{1}(\Omega)}M(\ep)|10^{M(\ep)}\ep|^3 \aleq -\log\ep |\ep^{1-\alpha}|^3,
\end{equation*}
which tends to zero as $\ep\rightarrow 0$. Similarly we can estimate $\vc{R}_\ep ( \vc{h}^\ep_1, \cdots, \vc{h}^\ep_{M(\ep)} )[ \bfphi ]$ with the help of \eqref{Rphi} and conclude that
\begin{equation} \label {error1}
	\| ( \bfphi - \vc{R}_\ep ( \vc{h}^\ep_1, \cdots, \vc{h}^\ep_{M(\ep)} )[ \bfphi ] )(t, \cdot) \|_{W^{1,p}(\R^d; \R^d)} \to 0 
	\ \mbox{uniformly for}\ t \in [0,T] \ \mbox{for any}\ 1 \leq p < \infty 
\end{equation} 
for any $\bfphi \in C^1_c([0, T) \times \Omega; \R^d),\ \Div \bfphi = 0$.

As for the time derivative, we use formula \eqref{TD} obtaining 
\begin{align}
&\partial_t \bfphi - \partial_t \vc{R}_\ep ( \vc{h}^\ep_1, \cdots, \vc{h}^\ep_{M(\ep)} )[ \bfphi ] = 
\partial_t \bfphi -  \vc{R}_\ep ( \vc{h}^\ep_1, \cdots, \vc{h}^\ep_{M(\ep)} )[ \partial_t \bfphi ]  \br &
+ \sum_{i=1}^{M(\ep)} \vc{R}_\ep (\vc{h}_1^{\ep}, \cdots, \vc{h}_{i-1}^{\ep}) \Big[ \Grad \vc{R}_{5^{i-1} \ep}(\vc{h}_i^{\ep}, \cdots, \vc{h}_{M(\ep)}^{\ep}) [\bfphi] \cdot \vc{Y}^\ep_i
\br &- \vc{R}_{5^{i-1} \ep}(\vc{h}_i^{\ep}) \left[ \Grad \vc{R}_{5^i \ep}(\vc{h}_{i+1}^{\ep}, \cdots, \vc{h}_{M(\ep)}^{\ep}) [\bfphi] \cdot \vc{Y}^\ep_i                           \right]  \Big],
\nonumber 
\end{align}
where, similarly to the above, 
\begin{equation} \label {error2}
	\| ( \partial_t \bfphi - \vc{R}_\ep ( \vc{h}^\ep_1, \cdots, \vc{h}^\ep_{M(\ep)} )[ \partial_t \bfphi ] )(t, \cdot) \|_{W^{1,p}(\R^d; \R^d)} \to 0 
	\ \mbox{uniformly in}\ t \in [0,T] \ \mbox{for any}\ 1 \leq p < \infty. 
\end{equation}
Finally, the second error term can be estimated with the help of \eqref{E12RR},
\begin{align} 
&\Big| \sum_{i=1}^{M(\ep)} \vc{R}_\ep (\vc{h}_1^{\ep}, \cdots, \vc{h}_{i-1}^{\ep})  \Big[ \Grad \vc{R}_{5^{i-1} \ep}(\vc{h}_i^{\ep}, \cdots, \vc{h}_{M(\ep)}^{\ep}) [\bfphi] \cdot \vc{Y}^\ep_i \br&
- \vc{R}_{5^{i-1} \ep}(\vc{h}_i^{\ep}) \left[ \Grad \vc{R}_{5^i \ep}(\vc{h}_{i+1}^{\ep}, \cdots, \vc{h}_{M(\ep)}^{\ep}) [\bfphi] \cdot \vc{Y}^\ep_i                          \right]  \Big] \Big| \br 
& \quad \quad \leq \mathds{1}_{ \cup_{i = 1}^N B_{10^{M(\ep)} \ep}(\vc{h}^\ep_i) }
\sum_{i=1}^{M(\ep)} |\vc{Y}^\ep_i | \left(  \Big| \vc{R}_\ep (\vc{h}_1^{\ep}, \cdots, \vc{h}_{i-1}^{\ep})   \Big[ \Grad \vc{R}_{5^{i-1} \ep}(\vc{h}_i^{\ep}, \cdots, \vc{h}_{M(\ep)}^{\ep}) [\bfphi] \Big] \Big| \right. \br &\quad \quad \quad \quad \quad \quad + 
\left. \vc{R}_\ep (\vc{h}_1^{\ep}, \cdots, \vc{h}_{i}^{\ep}) \Big[ \Grad \vc{R}_{5^i \ep}(\vc{h}_{i+1}^{\ep}, \cdots, \vc{h}_{M(\ep)}^{\ep}) [\bfphi] \Big] \right)  
\label{error3} 
\end{align}
where, by virtue of \eqref{E14bisR}, \eqref{E14R}, 
\begin{align} 
&\left\| \vc{R}_\ep (\vc{h}_1^{\ep}, \cdots, \vc{h}_{i-1}^{\ep})   \Big[ \Grad \vc{R}_{5^{i-1} \ep}(\vc{h}_i^{\ep}, \cdots, \vc{h}_{M(\ep)}^{\ep}) [\bfphi] \Big]  \right\|_{L^p(\R^d; \R^{d \times d})} \br  &\quad + 
\left\|  \vc{R}_\ep (\vc{h}_1^{\ep}, \cdots, \vc{h}_{i}^{\ep}) \Big[ \Grad \vc{R}_{5^i \ep}(\vc{h}_{i+1}^{\ep}, \cdots, \vc{h}_{M(\ep)}^{\ep}) [\bfphi] \Big]  \right\|_{L^p(\R^d; \R^{d \times d})}	
\br	&\quad \leq c(p)^{M(\ep)} 
\| \bfphi \|_{W^{1,p} (\R^d; \R^d)} \ \mbox{uniformly for}\ t \in (0,T)
\label{error4}
	\end{align}
and any $1 < p < \infty$.

\subsection{Convergence}

We know from estimates \eqref{u1}--\eqref{u2} that 
\begin{equation}\label{rhou}
\sqrt{\vre}\bu^{\ep} \mbox{ is bounded in }L^{\infty}(0,T;L^2(\Omega; \R^d)),
\end{equation}
\begin{equation*}
\bu^{\ep} \mbox{ is bounded in } L^2(0,T;W^{1,2}_0(\Omega; \R^d)).
\end{equation*}
Thus there exists $\bu\in L^{\infty}(0,T;L^2(\Omega)) \cap L^2(0,T;W^{1,2}_0(\Omega))$ such that, up to a subsequence,
\begin{equation}\label{con:ws}
 \sqrt{\vre}\bu^{\ep} \rightarrow \bu \mbox{ weak-${*}$ in }L^{\infty}(0,T;L^2(\Omega; \R^d)),
 \end{equation}
 \begin{equation}\label{con:w}
 \bu^{\ep} \rightarrow \bu \mbox{ weakly in }L^2(0,T;W^{1,2}_0(\Omega; \R^d)).
 \end{equation}

\medskip 

\subsubsection{Limit in the momentum equation}

Our ultimate goal is to perform the limit in the momentum equation \eqref{r2}, with the test function 
$\vc{R}_\ep ( \vc{h}^\ep_1, \cdots, \vc{h}^\ep_{M(\ep)} )[ \bfphi ]$.

\noindent \underline{\bf Viscous term.}  

In view of \eqref{con:w} and the error estimate \eqref{error1}, it is easy to see 

\begin{equation}\label{I1}
\int\limits_{0}^{T} \int\limits_{\Omega} \mathbb{S}(\Ds \bu^{\ep}): \Ds(\vc{R}_\ep ( \vc{h}^\ep_1, 
\cdots, \vc{h}^\ep_{M(\ep)} )[ \bfphi ]) \dx \dt
\to \int_0^T \intO{ \mathbb{S}(\Ds \bu):  \Ds \bfphi } \dt
\end{equation}
for any $\bfphi \in C^1_c([0,T) \times \Omega; \R^d)$, $\Div \bfphi = 0$.

\noindent\underline{\bf Convective term.}

As $\vre$ is bounded and the uniform bounds \eqref{u1}, \eqref{u2} hold, it is easy to check that 
\[
\vre \vue \otimes \vue \to \Ov{\bu \otimes \bu} \ \mbox{weakly in} \ L^p((0,T) \times \Omega; \R^{d \times d})
\]
for some $p > 1$. Consequently, in view of \eqref{error1}, 
\begin{equation} \label{I2}
\int\limits_{0}^{T} \int\limits_{\Omega} (\rho^{\ep}\bu^{\ep}\otimes \bu^{\ep}): \Grad(\vc{R}_\ep (\vc{h}^{\ep}_1, \cdots, \vc{h}^\ep_{M(\ep)})[\bfphi]) \dx \dt\rightarrow \int\limits_{0}^{T} \int\limits_{\Omega} \overline{(\bu\otimes \bu)}:\Grad \bfphi \dx \dt \mbox{ as }\ep\rightarrow 0
\end{equation}
for any  $\bfphi \in C^1_c([0,T) \times \Omega; \R^d)$, $\Div \bfphi = 0$.

\medskip 
\noindent \underline{\bf Time derivative.}
Our next goal is to establish the limit 
\begin{equation}\label{I3}
\int\limits_{0}^{T} \int\limits_{\Omega} \rho^{\ep}\bu^{\ep} \cdot\partial_t \vc{R}_\ep ( \vc{h}^\ep_1, \cdots, \vc{h}^\ep_{M(\ep)} )[ \bfphi ] \dx \dt \to   \int\limits_{0}^{T} \int\limits_{\Omega}  \bu \cdot\partial_t  \bfphi  \dx \dt.
\end{equation}
In view of the estimates \eqref{error2}, \eqref{error3} this amounts to show
\begin{equation} \label{I3bis}
\int_0^T \intO{ \mathds{1}_{ \cup_{i = 1}^N B_{10^{M(\ep)} \ep}(\vc{h}^\ep_i) } \rho^{\ep} \bu^{\ep} \cdot \sum_{i = 1}^N |\vc{Y}^\ep_i| |\vc{G}^\ep_i  |
	} \dt \to 0	,
	\end{equation} 
where, 
\begin{equation*}
\vc{G}^\ep_i= \vc{R}_\ep (\vc{h}_1^{\ep}, \cdots, \vc{h}_{i-1}^{\ep})   \Big[ \Grad \vc{R}_{5^{i-1} \ep}(\vc{h}_i^{\ep}, \cdots, \vc{h}_{M(\ep)}^{\ep}) [\bfphi] \Big] + \vc{R}_\ep (\vc{h}_1^{\ep}, \cdots, \vc{h}_{i}^{\ep}) \Big[ \Grad \vc{R}_{5^i \ep}(\vc{h}_{i+1}^{\ep}, \cdots, \vc{h}_{M(\ep)}^{\ep}) [\bfphi] \Big].
\end{equation*}
By virtue of \eqref{error4},
\begin{equation} \label{I4bis}
\left\| \vc{G}^\ep_i \right\|_{L^\infty(0,T; L^p(\Omega; \R^d ))} \leq c(p)^{M(\ep)}\| \bfphi \|_{L^{\infty}(0,T; W^{1,p} (\R^d; \R^d))}
\ \mbox{for any}\ 1 \leq p < \infty,\ i=1,\cdots, M(\ep).
\end{equation}

Let us start with the case $d=3$. 
In view of \eqref{error3} and uniform boundedness of the density, we have 
\begin{align} \label{I4}
&\left| \intO{\mathds{1}_{ \cup_{i = 1}^N B_{10^N \ep}(\vc{h}^\ep_i) } \rho^{\ep} \bu^{\ep} \cdot \sum_{i = 1}^N |\vc{Y}^\ep_i| \vc{G}^\ep_i } \right| \br 
&\quad \aleq c(\delta)c(p)^{M(\ep)} \sum_{i=1}^{M(\ep)} |\vc{Y}^\ep_i (t) | \| \vue \|_{L^6(\Omega; \R^d)} | \cup_{j=1}^N B_{10^N \ep}(\vc{h}^\ep_j) |^{\frac{5}{6} - \delta } \ \mbox{for any} \ \delta > 0.
\end{align}
Let us write 
\begin{align}
\sum_{i=1}^{M(\ep)} |\vc{Y}^\ep_i (t) | \| \vue \|_{L^6(\Omega; \R^d)}c(p)^{M(\ep)} | \cup_{j=1}^N B_{10^N \ep}(\vc{h}^\ep_j) |^{\frac{5}{6} - \delta } \br = 
\sum_{i=1}^{M(\ep)} |\mathcal{S}^i_\ep|^{\frac{1}{6}} |\vc{Y}^\ep_i (t) | \| \vue \|_{L^6(\Omega; \R^d)} \left(c(p)^{M(\ep)} | \cup_{j=1}^N  B_{10^N \ep}(\vc{h}^\ep_j) |^{\frac{5}{6} - \delta }|\mathcal{S}^i_\ep|^{- \frac{1}{6}} \right).
\nonumber
\end{align}
We deduce from the uniform bounds \eqref{u2}, \eqref{u53} 
\[|\mathcal{S}^i_\ep|^{\frac{1}{6}}
\left 
  \| |\vc{Y}^\ep_i (t) | \| \vue \|_{L^6(\Omega; \R^d)} \right\|_{L^1(0,T)} \aleq 1.
\]
Now 
\[
c(p)^{M(\ep)}| \cup_{j=1}^{M(\ep)}  B_{10^{M(\ep)} \ep}(\vc{h}^\ep_j) |^{\frac{5}{6} - \delta }|\mathcal{S}^i_\ep|^{- \frac{1}{6}} \leq c(p)^{M(\ep)}{M(\ep)} (10^{M(\ep)} \ep)^ {\frac{5}{2} - 3\delta} \ep^ {-\beta/6}.
\]
 Let us take $\max\{c(p)^{M(\ep)},10^{M(\ep)}\} \sim \ep^{-\alpha}$, i.e, $M(\ep)\sim -{\alpha}\log\ep$ for some $\alpha\in (0,5/7)$. Then we have 
\[
c(p)^{M(\ep)}M(\ep) (10^{M(\ep)} \ep)^ {\frac{5}{2} - 3\delta} \ep^ {-\beta/6} \aleq -\ep^{-\alpha}\log \ep (\ep^ {1-\alpha})^ {\frac{5}{2} - 3\delta}\ep^ {-\beta/6}= -(\log \ep) (\ep^ {\frac{5}{2}-\frac{7\alpha}{2}-\frac{\beta}{6} - 3(1-\alpha)\delta}).
\]
Hence, as a consequence of hypothesis \eqref{m17}, i.e, by taking $\beta< 15-21\alpha$,
\[
\left(c(p)^{M(\ep)} | \cup_{j=1}^N  B_{10^N \ep}(\vc{h}^\ep_j) |^{\frac{5}{6} - \delta }|\mathcal{S}^i_\ep|^{- \frac{1}{6}} \right) \to 0
\ \mbox{as long as} \ \delta > 0 \ \mbox{is small enough.}
\]

The same result can be obtained in the case $d =2$ by means of the Sobolev embedding $W^{1,2} \subset L^p$ 
for any finite $p$.

\subsubsection{Limit in the convective term}

The only thing remaining is to establish the identity:
\begin{equation} \label{c13}
\Ov{ \bu \otimes \bu } =  \bu \otimes \bu.
\end{equation}
We consider the quantity 
\[
\bfphi = \psi(t) \vc{R}_\ep (\vc{h}^\ep_1, \cdots, \vc{h}^\ep_{M(\ep)} )(t) [\phi],\ 
\psi = \psi(t) \in C^1_c[0,T), \ \phi \in C^1_c(\Omega; \R^d),\ \Div \phi = 0, 
\]
as a test function in the momentum equation \eqref{r2}. It follows that the time distributional derivative 
of 
\[
	t \in [0,T] \mapsto \intRd{ \rho^{\ep} \bu^{\ep} \cdot \vc{R}_\ep (\vc{h}^{\ep}_1, \cdots, \vc{h}^\ep_{M(\ep)} ) [\phi] }
\]
belongs to $L^q(0,T)$ for any $1 \leq q < 2$ and Arzel\` a--Ascoli theorem yields
\begin{align}
	t \in [0,T] \mapsto \intRd{ \rho^{\ep} \bu^{\ep} \cdot \vc{R}_\ep (\vc{h}^{\ep}_1, \cdots, 
		\vc{h}^\ep_{M(\ep)} ) [\phi] }
	\ \mbox{is precompact in}\ C[0,T] \br \mbox{for any} \ \phi \in C^1_c(\Omega; \R^d),\ \Div \phi = 0.
	\label{c6}
	\end{align}

Next, as $\rho^{\ep} \bu^{\ep}$ is bounded in $L^\infty(0,T; L^2(\R^d; \R^d))$, we use the error estimates for the operator $\vc{R}_{\ep}(\vc{h}^\ep_1, \cdots, \vc{h}^\ep_{M(\ep)})$ established in \eqref{error1}
to show 
\[
 \intRd{ (\rho^{\ep} \bu^{\ep}) (t, \cdot) \cdot \Big( \vc{R}_\ep (\vc{h}^{\ep}_1, \cdots, \vc{h}^\ep_{M(\ep)}) (t) ] [\phi] - \phi \Big) } \to 0 \ \mbox{ in }\ L^\infty (0,T),
\]
which, combined with \eqref{c6}, implies 
\begin{equation} \label{c8}
	\intRd{ (\vre \vue) (t, \cdot) \cdot \phi } \to 
	\intRd{ \bu(t, \cdot) \cdot \phi } \ \mbox{in}\ L^\infty (0,T)\ \mbox{for any}\ \phi \in C^1_c(\Omega; \R^d),\ \Div \phi = 0.
\end{equation}

 Using a density argument, we deduce from \eqref{c8} that 
\begin{equation} \label{c9}
	\int_\Omega (\vre \vue) (t, \cdot)  \cdot \phi \ \dx  \to 
	\int_\Omega \bu(t, \cdot) \cdot \phi \ \dx \ \mbox{in}\ L^\infty (0,T), 
\end{equation}
for any $\phi \in L^2(\Omega; \R^d),\ \Div \phi = 0 \mbox{ in }\Omega, \ \phi\cdot \vc{n}=0 \mbox{ on }\partial\Omega$.

For a genaral $\phi \in C^1_c (\Omega; \R^d)$, consider its Helmholtz decomposition  in $\Omega$, 
\[
\phi = \vc{H}[ \phi] + \Grad \Psi \ \mbox{in}\ \Omega \mbox{ with }\Grad\Psi\cdot\vc{n}=0 \mbox{ on }\partial\Omega.
\]
Accordingly, we get 
\[
	\int_\Omega (\vre \vue) (t, \cdot)  \cdot \phi \ \dx = 
		\int_\Omega (\vre \vue) (t, \cdot)  \cdot \vc{H} [\phi]  \ \dx + 	\int_\Omega (\vre \vue) (t, \cdot) \cdot \Grad \Psi \ \dx,
\]
where, in accordance with \eqref{c8},
\[
		\int_\Omega (\vre \vue) (t, \cdot) \cdot \vc{H} [\phi]  \ \dx \to 
		\int_\Omega \bu(t, \cdot)  \cdot \vc{H} [\phi]  \ \dx	= \int_\Omega \bu (t, \cdot) \cdot  \phi  \ \dx	
\ \mbox{in}\ L^\infty(0,T).
\]
Moreover, as $\vue$ is solenoidal
\begin{align} 
		\int_\Omega (\vre \vue) (t, \cdot)  \Grad \Psi  \ \dx &= 
		\int_\Omega (\vre \vue - \vue) (t, \cdot) \cdot \Grad \Psi  \ \dx \br &= 
\sum_{i=1}^{M(\ep)} (\vr_{\mc{S}^i}^{\ep} - 1)		\int_{\mathcal{S}^i_\ep} \vue \cdot \Grad \Psi \ \dx
\nonumber
\end{align}
for a.a. $t \in (0,T)$. Thus it follows from the uniform bounds established in \eqref{u2} that 
\[
	\int_\Omega (\vre \vue) (t, \cdot)  \cdot \Grad \Psi  \ \dx \to 0 \ \mbox{in}\ L^2(0,T).
	\]
and we may infer that	
\[
\int_\Omega (\vre \vue) (t, \cdot)  \phi  \ \dx \to 
\int_\Omega \bu (t, \cdot)  \phi  \ \dx	
\ \mbox{in}\ L^2(0,T) \ \mbox{for any}\ \phi \in C_c^1(\Omega;\R^d).
\]
By density, we extend the conclusion to square integrable function, 
\begin{equation} \label{c10}
\int_\Omega (\vre \vue) (t, \cdot) \cdot \phi  \ \dx \to 
\int_\Omega \bu (t, \cdot) \cdot  \phi  \ \dx	
\ \mbox{in}\ L^2(0,T) \ \mbox{for any}\ \phi \in L^2(\Omega; \R^d).
\end{equation}
Equivalently, we can extend the function $\vue$ by zero in $\R^d \setminus \Omega$ and obtain:
\begin{equation} \label{c11}
	\vre \vue \to \bu \ \mbox{in}\ L^2(0,T; L^2_{\rm weak} (\R^d; \R^d)).
\end{equation}
Since $L^2_{\rm weak}(K; \R^d)$ is compactly embedded in the dual $W^{-1,2}(K; \R^d)$ for any compact $K \subset \R^d$, the desired conclusion 
\begin{equation} \label{c12} 
	\int_0^T \intRd{ \vre \vue \otimes \vue : \Grad \bfphi } \dt \to 
\int_0^T \intRd{   \bu \otimes \bu : \Grad \bfphi } \dt \ \mbox{for any}\ \bfphi \in C^1_c([0,T) \times \Omega; \R^d)
\end{equation}
follows. 

We infer from the above discussion on the passing to the limit as $\ep\rightarrow 0$ that the limit velocity of $\bu^{\ep}$ is given by $\bu$ where
$\bu$ is a weak solution of the Navier--Stokes system \eqref{m1}--\eqref{m3} satisfying the energy inequality \eqref{m26}.

\def\cprime{$'$} \def\ocirc#1{\ifmmode\setbox0=\hbox{$#1$}\dimen0=\ht0
	\advance\dimen0 by1pt\rlap{\hbox to\wd0{\hss\raise\dimen0
			\hbox{\hskip.2em$\scriptscriptstyle\circ$}\hss}}#1\else {\accent"17 #1}\fi}



\end{document}